\documentclass[12pt]{amsart}

\usepackage{graphicx}

\usepackage{subcaption} 

\usepackage{tikz-cd}
\usepackage{hyperref}
\hypersetup{bookmarks=true,
	unicode=true,
	colorlinks=true,
	citecolor=black,
	linkcolor=black,
	urlcolor=black,
	plainpages=false,
	pdfpagelabels=true}

 \usepackage{url}	
 \allowdisplaybreaks 

\usepackage{pgf}

\usepackage{xcolor}

\usepackage{comment} 

\usepackage{tikz}
\usetikzlibrary{arrows}
\usetikzlibrary{decorations.markings,arrows,automata,arrows,backgrounds,snakes}
\usepackage{tikz-cd} 
\usepackage{pgf}
\usetikzlibrary{babel}

\usepackage{appendix}

\usepackage{mathrsfs}
\usepackage{dsfont}

\usepackage{mathtools}

\usepackage[shortlabels]{enumitem}

\newtheorem{theorem}{Theorem}[section]
\newtheorem{proposition}[theorem]{Proposition}
\newtheorem{lemma}[theorem]{Lemma}
\newtheorem{corollary}[theorem]{Corollary}

\theoremstyle{definition}
\newtheorem{definition}[theorem]{Definition}

\theoremstyle{remark}
\newtheorem{remark}[theorem]{Remark}

\usepackage{color}
\definecolor{darkgreen}{cmyk}{1,0,1,.2}
\definecolor{m}{rgb}{1,0.1,1}

\newdimen\theight
\def\TeXref#1{%
             \leavevmode\vadjust{\setbox0=\hbox{{\tt
                     \quad\quad  {\small \textrm #1}}}%
             \theight=\ht0
             \advance\theight by \lineskip
             \kern -\theight \vbox to
             \theight{\rightline{\rlap{\box0}}%
             \vss}%
             }}%

\begin{document}

 \title{A Combinatorial Study of the Fixed Point Index}
 \thanks{This research was partially supported by the grant PID2020-114474GB-I00 (AEI/FEDER, UE)
}

\author[J. Álvarez \and A. Majadas-Moure \and D. Mosquera
     ]{%
	Jesús A. Álvarez López \and Alejandro O. Majadas-Moure \and David Mosquera-Lois
}

\address{
         Jesús A. Álvarez López\\
         Departamento de Matemáticas, Universidade de Santiago de Compostela,
           SPAIN}
         \email{jesus.alvarez@usc.es}
              
\address{
		 Alejandro O. Majadas-Moure \\
		 Departamento de Matemáticas, Universidade de Santiago de Compostela, SPAIN}
		 \email{alejandro.majadas@usc.es}
		
 \address{
           	 David Mosquera-Lois \\
              Departamento de Matemáticas,  Universidade de Vigo, SPAIN }\email{david.mosquera.lois@uvigo.gal}
		

\begin{abstract} 
We introduce a theory of integration with respect to the fixed point index, offering a substantial improvement over previous approaches based on the Lefschetz number. This framework eliminates several restrictive assumptions—such as the need for definability, openness, or f-invariance of subspaces—thereby allowing broader applicability. We also present a natural combinatorial adaptation of the fixed point index that extends the combinatorial Lefschetz number. This extension yields new topological and homotopical invariance results and facilitates the integration of real-valued functions with respect to fixed points.
\end{abstract}



\maketitle

\section{Introduction}

The Lefschetz number is a classical topological invariant associated with a continuous map \( f: X \to X \) on a compact polyhedron, defined as
\[
\Lambda(f) = \sum_{i \geq 0} (-1)^i \operatorname{tr}(f_*: H_i(X) \to H_i(X)),
\]
where \( f_* \) denotes the induced morphism in rational homology. This invariant not only generalizes the Euler--Poincaré characteristic (recovered by taking \( f = \mathrm{id} \)), but also plays a central role in fixed point theory by guaranteeing the existence of fixed points under certain conditions.

A more flexible tool in this context is the fixed point index, which extends the Lefschetz number to broader classes of spaces and maps. It is defined for a larger family of topological spaces and satisfies the relation \( \Lambda(f) = i(X, f, X) \), making it fundamental to fixed point theory in both compact and non-compact settings (see \cite{Brown} for details).

In recent work, combinatorial generalizations of these invariants have been developed. The combinatorial Euler characteristic (see \cite{B-G1, C-G-R}) and combinatorial Lefschetz number (see \cite{M-M1, M-M2}) were introduced in the context of o-minimal and definable structures. These notions allow integration over non-compact or non-open subspaces, making them suitable for applications such as sensor networks.

A key advancement in this direction is the combinatorial Lefschetz number defined for definable, \( f \)-invariant (non necessarily compact) subspaces of simplicial complexes introduced in \cite{M-M1}. This extension presents two notable advantages: first, it yields a more general fixed point theorem (see \cite[Theorem~4.1]{M-M1}) than the classical Lefschetz fixed point theorem---especially when restricted for homeomorphisms; second, it provides a finer measure of fixed points, as illustrated by \cite[Example 4.3]{M-M1}:

Consider the 2-sphere \( S^2 \) and a homeomorphism \( f: S^2 \to S^2 \) that reflects over the equator. The classical Lefschetz number in this case is zero, despite every point on the equator being a fixed point. However, if we model \( S^2 \) as the suspension \( \Sigma X_m \) of a regular \( m \)-gon (with \( X_m \) an empty polygon of \( m \) edges), and define a subcomplex \( U_m \) by removing the open edges along the equator, then \( U_m \) is \( f \)-invariant. The combinatorial Lefschetz number \( \Lambda(f, U_m)_{X_m} \) then correctly counts the number of fixed vertices:
\[
\Lambda(f, U_m)_{X_m} = m.
\]
This example illustrates the refined resolution of fixed point data provided by the combinatorial framework---where the classical theory fails to detect fixed points due to cancellation, the combinatorial Lefschetz number recovers the precise count.

Despite these advances, prior definitions still rely on restrictive assumptions: subspaces must often be definable and $f$-invariant ($f(A)=A$), and maps are usually required to be homeomorphisms. These limitations obstruct the development of a general integration theory and fixed point calculus.

In this paper, we address these limitations by introducing a combinatorial fixed point index---a natural generalization of both the classical index and the combinatorial Lefschetz number. This new framework allows us to extend integration theory to real-valued functions without requiring openness, definability, or $f$-invariance (see Section \ref{int real}); to formulate new topological (Theorem~\ref{thm: inv top num comb} -generalizing \cite[Remark A.7]{McCrory}-) and homotopical (Corollary~\ref{inv homotopica}) invariance results; and to provide a foundation for measuring fixed point ``density'' in combinatorial and o-minimal contexts. This new index enables integration over broader families of spaces and maps, including those where $f$ is not a homeomorphism and where subspaces need not be invariant or open. It also leads to a Fubini-type formula (Theorem~\ref{especie de fubini}) for fiber bundles, further expanding the algebraic-topological tools available for studying fixed point data.

At the heart of this approach is a simple but powerful idea: define the combinatorial fixed point index of a subspace $A \subset X$, not by direct evaluation over \( A \), but by applying the classical index to its interior $\mathring{A}$, whenever $f$ has no fixed points on the boundary. This allows us to extend the fixed point calculus while preserving desirable properties such as additivity and homotopy invariance.

The structure of the paper is as follows. In Section \ref{seccion primeros resultados}, we recall the necessary background on o-minimal structures and definable spaces, and we introduce an axiomatic framework for the combinatorial Lefschetz number, culminating in a generalization: the combinatorial fixed point index. This section also establishes topological and homotopical invariance results for this index. In Section \ref{sec:integration_comb_index} we develop a theory of integration with respect to the combinatorial fixed point index, proving key properties such as a product formula and a Fubini-type theorem for fiber bundles. In Section \ref{int real}, the combinatorial fixed point index enables to extend the integration theory to real-valued functions, defining suitable Riemann-type sums and proving their convergence. This construction enables fixed point integration beyond characteristic functions, without assuming openness, definability, or $f$-invariance.

\section{The combinatorial Lefschetz number and the combinatorial fixed point index}\label{seccion primeros resultados}

\subsection{Definable structures}

We introduce the definable structures we will work with. For a more detailed account on $o$-minimal topology we refer the reader to \cite{Dries}. Recall that an \textit{$o$-minimal} structure over $\mathbb{R}$ is a collection $\mathscr{A}=\{\mathscr{A}_n\}_{n\in \mathbb{N}}$ so that the following properties hold:
\begin{enumerate}
\item $\mathscr{A}_n$ is an algebra of subsets of $\mathbb{R}^n$ for each $n\in \mathbb{N}$.
\item The family $\mathscr{A}$ is closed with respect to cartesian products and canonical projections.
\item The subset $\{(x,y)\in\mathbb{R}^2, x<y\}$ is in $\mathscr{A}_2$.
\item The family $\mathscr{A}_1$ consists of all finite unions of points and open intervals of $\mathbb{R}$.
\item Every $\mathscr{A}_n$ contains all the algebraic subsets of $\mathbb{R}^n$.
\end{enumerate}
Given an o-minimal structure, we will say that a set $A\subset \mathbb{R}^n$ is \textit{definable} if $A\in \mathscr{A}_n$.

A {\em generalized simplicial complex} is a finite collection $K $ of open simplices in $\mathbb{R}^n$. We illustrate a generalized simplicial subcomplex in Figure \ref{fig:generalized_simplicial_complex}, where, as we will do along the paper, by abuse of notation, we will denote by $X$ both the geometric realization and the underlying  simplicial complex. In the rest of this paper, all the complexes will be finite.
\begin{figure}[h!]
    \centering
\includegraphics[width=0.25\linewidth]{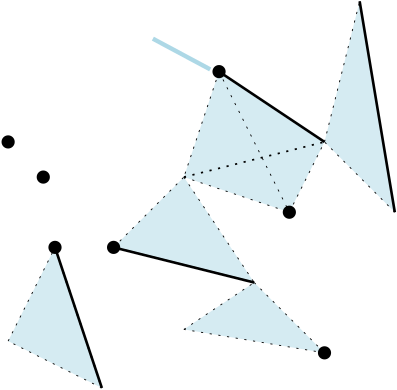}
    \caption{Generalized simplicial complex.}
\label{fig:generalized_simplicial_complex}
\end{figure}

We will use o-minimal structures that contain the semi-linear sets so that, in this way, the generalised simplicial complexes are definable. Moreover, we have the following triangulation theorem:

\begin{theorem}[{Definable triangulation theorem \cite{Dries}}] \label{thm:triangulation} 
	Let $X\subset \mathbb{R}^n$ be a definable set and let $\{X_i\}_{i=1}^{m}$ be a finite family of definable subsets of $X$. Then there exists a definable triangulation of $X$ compatible with the collection of subsets.
\end{theorem}
Once introduced the o-minimal structures, we can define the combinatorial Lefschetz number. This non-trivial definition and its properties were presented in \cite{M-M1}. We also recommend \cite[Section 2.2]{M-M2} for a quick lecture of the main ideas behind the combinatorial Lefschetz number. Given a simplicial complex $X$, a homeomorphism $f:X\rightarrow X$ and a definable subset $U\subset X$ which is $f$-invariant ($f(U)=U$), we denote the combinatorial Lefschetz number by $\varLambda(U,f)_X$ (for the usual Lefschetz number we will write $\varLambda(f)$).

\subsection{Axioms for the combinatorial Lefchetz number}  In this section, we characterize the combinatorial Lefschetz number (see \cite{M-M1, M-M2} for the definition) from an axiomatic point of view that will allow us to obtain a generalization: the combinatorial fixed point index. This may be seen as an extension of Arkowitz and Brown's work (\cite{A-B}) to non-necessarily compact settings. Note that the most important difference between the Lefschetz number and the combinatorial Lefchetz number is that the second one is additive even if the spaces are not compact (see \cite[Proposition 4.4]{M-M1}):

\begin{proposition}[Additivity of the combinatorial Lefschetz number]\label{lemma:additivity_combinatorial_lefschetz_number}
    Let $X$ be a definable cellular complex and let $U$ and $V$ be two disjoint definable and $f$-invariant subsets. Then, 
    \begin{equation*}
    \varLambda( U\cup V,f)_X=\varLambda(U,f)_X+\varLambda(V,f)_X.
    \end{equation*}
\end{proposition}

Let $\mathcal{D}$ the family of triples $(X,f,U)$, where $X$ is a definable compact space, $f:X\rightarrow X$ is a homeomorphism and $U\subset X$ is a definable $f$-invariant ($f(U)=U$) subspace of $X$.

\begin{definition}\label{def clase homeo}
    We will say that two triples $(X,f,U)$ and $(Y,g,V)$ of $\mathcal{D}$ \textit{belong to the same homeomorphism class} if there exists a homeomorphism $h:X\rightarrow Y$ such that $h_{|U}:U\rightarrow V$ is a homeomorphism and the diagram
    \[\begin{tikzcd}
	X & X \\
	Y & Y
	\arrow["f", from=1-1, to=1-2]
	\arrow["h"', from=1-1, to=2-1]
	\arrow["h", from=1-2, to=2-2]
	\arrow["g", from=2-1, to=2-2]
\end{tikzcd}\]
commutes.
 The homeomorphism class of triples define a equivalence relation in $\mathcal{D}$. We will denote by $\mathcal{C}$ the quotient of $\mathcal{D}$ through this relation.
\end{definition}

\begin{theorem}\label{thm:axioms_combinatorial_lefschetz_number}
    The combinatorial Lefschetz number is the only operator $\lambda$ from $\mathcal{C}$ to the integers satisfying the following two axioms:
    \begin{enumerate}
        \item If $U$ is compact, then $\lambda(X,f,U)=\varLambda(f_{|U})$.
        \item If $U, V\subset X$ are disjoint $f$-invariant definable subsets, then
        \begin{equation*}
            \lambda(X,f,U\cup V)=\lambda(X,f,U)+\lambda(X,f,V).
        \end{equation*}
    \end{enumerate}
\end{theorem}
\begin{proof}
    The combinatorial Lefschetz number satisfies both axioms (see \cite{M-M1} for details and notice that, from \cite[Definition 3.4]{M-M1}, the combinatorial Lefschetz number will be well defined in $\mathcal{C}$ and not only in $\mathcal{D}$).    
    
    Let $U$ be a definable $f$-invariant subset of $X$, where $f:X\rightarrow X$ is a homeomorphism. Now, we must check that if $\lambda:\mathcal{C}\rightarrow\mathbb{Z}$ satisties these two axioms, then $\lambda(X,f,U)=\varLambda(U,f)_X$. Let $(Y,g,K)$ be in the class of $(X,f,U)$ in $\mathcal{C}$, where $K$ is a generalized subcomplex of $Y$ (the existence of $(Y,g,K)$ is given by Theorem~\ref{thm:triangulation}). Since $g$ is a homeomorphism and $K$ is $g$-invariant, we have that $\overline{K}$, $\overline{K}\setminus K$,... are also $g$-invariant. They are also definable because the closure of a definable set is definable \cite[Lemma~3.4, Chapter~1]{Dries} and the definable sets form an algebra of subsets.

    Since $\lambda$ satisfies the second axiom we have:
    \begin{align*}
        &\lambda(Y,g,\overline{K})=\lambda(Y,g,K)+\lambda(Y,g,\overline{K}\setminus K),\\
        &\lambda(Y,g,\overline{\overline{K}\setminus K})=\lambda(Y,g,\overline{K}\setminus K)+\lambda\bigl(Y,g,(\overline{\overline{K}\setminus K}))\setminus (\overline{K}\setminus K)\bigr),\\
        &\lambda\bigl(Y,g,\overline{(\overline{\overline{K}\setminus K)}\setminus (\overline{K}\setminus K})\bigr)=\ldots,\\
        &\ldots
    \end{align*}
    First, note that, since $\lambda$ also satisfies the second axiom, the left-hand side of each of these equalities agrees with the Lefschetz number, that is:
        \begin{align*}
        &\varLambda(g_{|\overline{K}})=\lambda(Y,g,K)+\lambda(Y,g,\overline{K}\setminus K),\\
        &\varLambda(g_{|\overline{\overline{K}\setminus K}})=\lambda(Y,g,\overline{K}\setminus K)+\lambda\bigl(Y,g,(\overline{\overline{K}\setminus K})\setminus (\overline{K}\setminus K)\bigr),\\
        &\varLambda(g_{|(\overline{\overline{\overline{K}\setminus K})\setminus (\overline{K}\setminus K)}})=\ldots,\\
        &\ldots
    \end{align*}
    Secondly, notice that, on each step, the dimension of the last term of the spaces in the right-hand side of each equality decreases at least one, so if the dimension of $K$ is $m$, after $m$ steps (or possibly less) the space of the last term of the $m^{th}$ equality must consist of a finite union of points (or the empty set). Since a finite union of points is compact, $\lambda$ must agree there with the Lefschetz number. Finally, since the combinatorial Lefschetz number also satisfies these two axioms, the same sequence of equalities applies if we replace $\lambda$ by the combinatorial Lefschetz number. The equality between $\lambda$ and the combinatorial Lefschetz number is then automatic by rising the dimension. \qedhere

\end{proof}
\subsection{The Combinatorial Fixed Point Index}
We introduce the novel notion of combinatorial fixed point index. 

\begin{definition}\label{def indice comb}
    Let $X$ be a simplicial complex, $f:X\rightarrow X$ a continuous map and $A\subset X$ a subspace such that $f$ has no fixed points in $\overline{A}\setminus \mathring{A}$. The \textit{combinatorial fixed point index} is defined as:
    \begin{equation*}
        i_c(X,f,A)=i(X,f,\mathring{A}),
    \end{equation*}
    where $i(X,f,\mathring{A})$ denotes the fixed point index \cite[Chapter IV]{Brown}.
\end{definition}

We study some properties of the combinatorial fixed point index.

\begin{proposition}\label{pto fijo indice comb}
    Let $X$ be a simplicial complex, $A\subset X$ a subspace and $f:X\rightarrow X$ a continuous map such that $f$ has no fixed points in $\overline{A}\setminus\mathring{A}$. Then, if $i_c(X,f,A)\neq 0$, $f$ has a fixed point in $\mathring{A}$.
\end{proposition}
\begin{proof}
    Since $i_c(X,f,A)=i(X,f,\mathring{A})\neq 0$, by \cite[IV.A Corollary 1]{Brown}, $f$ must have a fixed point in $\mathring{A}$.
\end{proof}
\begin{remark}
    Proposition \ref{pto fijo indice comb} provides, due to Theorem~\ref{igualdad indice numero combinatorio}, another proof of \cite[Theorem 4.1]{M-M1} when the boundary of $A$ is fixed point free.
\end{remark}
The combinatorial index satisfies the four first axioms of the index (see \cite[IV.A]{Brown}). The fifth one (the Commutative Axiom) \cite[IV.A]{Brown} is only satisfied on open sets. Recall the definition of $\mathcal{C}$ from Definition \ref{def clase homeo}.

\begin{theorem}
     The combinatorial index is the only operator $\lambda:\mathcal{C}\rightarrow \mathbb{Z}$ which satisfies the following axioms:
     \begin{enumerate}
         \item[(CI1)] Localization Axiom: If $f,g:X\rightarrow X$ are maps such that they have no fixed point in $\overline{A}\setminus\mathring{A}$ and such that $f(x)=g(x)$ for all $x\in \overline{A}$, then $i_c(X,f,A)=i_c(X,g,A)$.
         \item[(CI2)] Homotopy Axiom: If $F: X\times I\rightarrow X$ is a homotopy between $f_0$ and $f_1$ without fixed points in $\overline{A}\setminus \mathring{A}$, then $i_c(X,f_0,A)=i_c(X,f_1,A).$
         \item[(CI3)] Additivity Axiom: If $A$ is a subspace such that $f$ doesn't have fixed points in $\overline{A}\setminus \mathring{A}$ and $A_1,\ldots, A_s\subset A$ are disjoint subspaces such that all the fixed points of $f$ in $\overline{A}$ are in $\mathring{A}_1\cup\ldots\cup \mathring{A}_s$ and $f$ doesn't have any fixed point in $\overline{A_j}\setminus \mathring{A}_j$ for all $j=1,\ldots, s$, then
        $i_c(X,f,A)=\sum_j i_c(X,f,A_j).$
         \item[(CI4)] Normalization Axiom: $i_c(X,f,X)=\varLambda(f).$
         \item[(CI5)] Commutativity Axiom: If $X$ and $Y$ are simplicial complexes, $A\subset X$ is an open subspace of $X$ and $f: X\rightarrow Y$ and $g: Y \rightarrow X$ are maps such that $g\circ f$ does not have any fixed point in $\overline{A}\setminus A=\overline{A}\setminus\mathring{A}$, then $i_c(X,g\circ f,A)=i_c(Y,f\circ g, g^{-1}(A)).$
     \end{enumerate}
\end{theorem}

\begin{proof}
We begin by proving that the combinatorial index satisfies the axioms. The idea is that the usual fixed point index satisfies its corresponding axioms \cite[IV.A]{Brown} and that the combinatorial index of $(X,f,A)$ is the usual index of $(X,f,\mathring{A})$. Let us see, for instance, CI1 (the other axioms follow analogously).

        [(CI1)] If both maps, $f$ and $g$, agree in the closure of the interior of $A$, hence
        \begin{equation*}
            i_c(X,f,A)=i(X,f,\mathring{A})=i(X,g,\mathring{A})=i_c(X,g,A).
        \end{equation*}
    It remains to check that if an operator $\lambda:\mathcal{C}\rightarrow \mathbb{Z}$ satisfies the axioms, then it is the combinatorial index.  
    If $\lambda$ satisfies the axioms, then, by \cite[Chapter IV]{Brown}, $\lambda(X,f,U)=i(X,f,U)$ if $U$ is open in $X$. Let $(X,f,A)\in \mathcal{C}$. Since $f$ has no fixed points in $\overline{A}\setminus\mathring{A}$, by the Additivity Axiom,
    \begin{equation*}
    \lambda(X,f,A)=\lambda(X,f,\mathring{A})=i(X,f,\mathring{A})=i_c(X,f,A). \qedhere
    \end{equation*}
\end{proof}

We prove that under certain hypotheses, the combinatorial Lefschetz number is the combinatorial index.

\begin{theorem}\label{igualdad indice numero combinatorio}
    Let $X$ be a simplicial complex, $f:X\rightarrow X$ a homeomorphism and $A\subset X$ a definable $f$-invariant subspace such that $f$ has no fixed points in $\overline{A}\setminus\mathring{A}$. Then
    \begin{equation*}
        \varLambda(A,f)_X=i_c(X,f,A).
    \end{equation*}
\end{theorem}
\begin{proof}
    We have the chain of equalities:
    \begin{equation*}
        \varLambda(f_{|\overline{A}})=i(\overline{A},f_{|\overline{A}},\overline{A})=i(\overline{A},f_{|\overline{A}},\mathring{A})=i(X,f,\mathring{A})=i_c(X,f,A),
    \end{equation*}
    where the first equality comes from the Normalization Axiom of the index, the second one comes from the Additivity Axiom and the third one comes from the topological invariance of the fixed point index \cite[Corllary 7.2]{ONeill}. The fourth equality is just the definition of the combinatorial index.

    By Theorem \ref{thm:axioms_combinatorial_lefschetz_number} (1) and the additivity of the combinatorial Lefschetz number (Theorem \ref{thm:axioms_combinatorial_lefschetz_number} (2)), we have:
    \begin{equation}\label{igualdad numero comb y normal_p}
        \varLambda(f_{\overline{A}})=\varLambda(\overline{A},f_{\overline{A}})_X=\varLambda(A,f)_X+\varLambda(\overline{A}\setminus A, f)_X.
    \end{equation}
    Since $f$ has not fixed points in $\overline{A}\setminus\mathring{A}$, then it has not fixed points in $\overline{\overline{A}\setminus A}$. Therefore, the fixed point theorem of the combinatorial Lefschetz number \cite[Theorem 4.1]{M-M1} implies $\varLambda(\overline{A}\setminus A,f)_X=0$ and as a consequence, from Equation \ref{igualdad numero comb y normal_p} we obtain:
    \begin{equation}\label{igualdad numero comb y normal}
        \varLambda(f_{\overline{A}})=\varLambda(\overline{A},f_{\overline{A}})_X=\varLambda(A,f)_X. \qedhere
    \end{equation}
\end{proof}

Let us present now some local computation and homotopy invariance properties of the combinatorial Lefschetz number that are consequences of Theorem~\ref{igualdad indice numero combinatorio}.

\begin{corollary}
    Let $X$ be a simplicial complex, $f,g:X\rightarrow X$ homeomorphisms and $A\subset X$ an $f$-invariant and $g$-invariant definable subspace such that $f$ and $g$ do not have any fixed points in $\overline{A}\setminus\mathring{A}$.
    \begin{enumerate}
        \item[(1)] \label{cor igual num comb y normal} The combinatorial Lefschetz number of $f$ in $A$ equals the Lefschetz number of $f_{|\overline{A}}$ and equals $\varLambda(\mathring{A},f)_X$, i.e: $$\varLambda(f_{\overline{A}})=\varLambda(A,f)_X=\varLambda(\mathring{A},f)_X$$
        \item[(2)] \label{inv homotopica} If $f_{|\overline{A}}$ and $g_{|\overline{A}}$ are homotopic, then 
    \begin{equation*}
    \varLambda(A,f)_X=\varLambda(A,g)_X.
    \end{equation*}
        \item[(3)] \label{inv homotopica} If $F: X\times I\rightarrow X$ is a homotopy between $f$ and $g$ without fixed points in $\overline{A}\setminus\mathring{A}$. Then
    \begin{equation*}
    \varLambda(A,f)_X=\varLambda(A,g)_X.
    \end{equation*}
    \end{enumerate}
\end{corollary}

\begin{proof}
\begin{enumerate}
    \item[(1)] The first equality follows from Equation (\ref{igualdad numero comb y normal}). For the second, since $f$ is a homeomorphism and $A$ is $f$-invariant, so is $\mathring{A}$. Moreover, $\mathring{A}$ is definable by \cite[Lemma (3.4)]{Dries}. Therefore, $\varLambda(\mathring{A},f)_X$ is defined.  From Theorem~\ref{igualdad indice numero combinatorio} it follows that
    \begin{equation*}
    \varLambda(A,f)_X=i_c(X,f,A)=i(X,f,\mathring{A})=i_c(X,f,\mathring{A})=\varLambda(\mathring{A},f)_X.
    \end{equation*}
    \item[(2)]   From Corollary~\ref{cor igual num comb y normal} (1), we have $\varLambda(A,f)_X=\varLambda(f_{|\overline{A}})$ and $\varLambda(A,g)_X=\varLambda(g_{|\overline{A}})$. The result follows then from the homotopic invariance of the Lefschetz number. 
    \item[(3)]  The result follows from Theorem~\ref{igualdad indice numero combinatorio} and from the Homotopy Axiom of the combinatorial fixed point index. \qedhere
\end{enumerate}
    
\end{proof}


\subsection{Topological invariance of the combinatorial Lefschetz number}
From the definition of combinatorial Lefschetz number (see \cite{M-M1}) follows a first topological invariance result. It is a consequence of \cite[Definition 3.4]{M-M1}, since the triangulation of $(X,\overline{U},U)$ is also a triangulation of $(Y,\overline{V},V)$ and because the definition of combinatorial Lefschetz number does not depend on the choice of the triangulation (\cite[Theorem 3.9 (b)]{M-M1}).
\begin{theorem}\label{primera inv top}
    Let $X$, $Y$ be definable compact spaces, $f:X\rightarrow X$, $g:Y\rightarrow Y$ homeomorphisms and $U\subset X$, $V\subset Y$ definable $f$-invariant (resp. $g$-invariant) subspaces of $X$ (resp. $Y$). With the terminology of Definition~\ref{def clase homeo}, if $(X,f,U)$ and $(Y,g,V)$ belong to the same homeomorphism class, then $\varLambda(U,f)_X=\varLambda(V,g)_Y$.
\end{theorem}


In \cite[Corollary 3.4]{M-M2}, we obtained a result on the topological invariance of the combinatorial Lefschetz number. Nevertheless, the hypotheses were very rigid. Here, we are flexibilizing the hypotheses by proving the topological invariance of the combinatorial Lefschetz number when the map does not have any fixed point in the boundary of the definable subset where we are computing the combinatorial Lefschetz number.

\begin{theorem}\label{thm: inv top num comb}
     Let $A$ and $B$ be definable sets and let $f:A\rightarrow A$ and $g:B\rightarrow B$ be homeomorphism such that there exist a definable homeomorphism $h:A\rightarrow B$ making the diagram
    \[\begin{tikzcd}
	A & A \\
	B & B
	\arrow["f", from=1-1, to=1-2]
	\arrow["h"', from=1-1, to=2-1]
	\arrow["h", from=1-2, to=2-2]
	\arrow["g", from=2-1, to=2-2]
\end{tikzcd}\]
commute. Suposse also that there exist simplicial complexes $X$ and $Y$ such that $A\subset X$, $B\subset Y$, and that $f$ and $g$ can be extended to homeomorphisms $f':X\rightarrow X$ and $g':Y\rightarrow Y$ ($X$ and $Y$ need not to be homeomorphic). Then, if $f'$ (resp. $g'$) doesn't have any fixed point in $\overline{A}\setminus\overset{\circ}{A}$ (resp. $\overline{B}\setminus\overset{\circ}{B}$) (here the closures are taken by considering $A$ and $B$ as subspaces of $X$ and $Y$), we have
    \begin{equation*}
        \varLambda(A,f')_X=\varLambda(B,g')_Y.
    \end{equation*}
\end{theorem}
\begin{proof}
Due to the Triangulation Theorem (Theorem~\ref{thm:triangulation}) and Theorem~\ref{primera inv top}, we can assume $A$ and $B$ are generalized subcomplexes of $X$ and $Y$. Moreover, by \cite[Remark 3.5]{M-M1}, we have $\varLambda(A,f')_X=\varLambda(A,f'_{|\overline{A}})_{\overline{A}}$ and $\varLambda(B,g')_Y=\varLambda(B,g'_{|\overline{B}})_{\overline{B}}$ (taking the closures of $A$ and $B$ in $X$ and $Y$). We must then check that $\varLambda(A,f'_{|\overline{A}})_{\overline{A}}=\varLambda(B,g'_{|\overline{B}})_{\overline{B}}$.

By Theorem~\ref{igualdad indice numero combinatorio}, we have $\varLambda(A,f'_{|\overline{A}})_{\overline{A}}=i_c(\overline{A},f'_{|\overline{A}},A)$ and $\varLambda(B,g'_{|\overline{B}})_{\overline{B}}=i_c(\overline{B},g'_{|\overline{B}},B)$.

Denote by $\mathrm{Fix}(f'_{|\overline{A}})$ the set of fixed points of $f'_{|\overline{A}}$. Now, since $f'_{|\overline{A}}$ has no fixed points in $\overline{A}\setminus \overset{\circ}{A}$, we can take a open set $U$ close enough to $\mathrm{Fix}(f'_{|\overline{A}})$ such that $U\cup f'_{|\overline{A}}(U)$ is contained in a simplicial subcomplex $K$ of a subdivision of $\overline{A}$ satisfying that the underlying space of $K$ is contained in the one of $A$.

Now, since $h$ was definable, by the Triangulation Theorem and Theorem~\ref{primera inv top}, we can assume that $h(K)$ is a simplicial subcomplex of a subdivision of $\overline{B}$ such that its underlying space is contained in the one of $B$. 

In this case, \cite[Corollary 7.2]{ONeill} applies and we have 
\begin{equation*}
    i_c(\overline{A},f'_{|\overline{A}},A)=i_c(\overline{A},f'_{|\overline{A}},U)= i_c(\overline{B},g'_{|\overline{B}},h(U))=i_c(\overline{B},g'_{|\overline{B}},B),
\end{equation*} 
where the first and the third equalities are due to the Additivity Axiom of the combinatorial fixed point index. So the result follows.
\end{proof}
\begin{remark}
    In Theorem \ref{thm: inv top num comb} the complexes $X$ and $Y$ need not to be homeomorphic (this was the main restriction of Theorem \ref{primera inv top}). In fact, the existence of $X$, $Y$, $f'$ and $g'$ is now only used to make sense the notion of combinatorial Lefschetz number.
    
    Note that $h$ is required to be definable only to guarantee that $h(K)$ can be considered as a simplicial complex. This is important because \cite[Theorem 2.5, L5]{ONeill} requires the fixed points of $f'_{|\overline{A}}$ and $g'_{|\overline{B}}$ to belong to simplicial complexes that are homeomorphic under $h$. As a consequence, a stronger topological invariance result of the fixed point index might allow us to require $h$ only to be a homeomorphism. 
\end{remark}

\section{Integration with respect to the Combinatorial Fixed Point Index}
\label{sec:integration_comb_index}

In \cite{M-M1}, an integration theory with respect to the Lefschetz number was presented. For that, we considered a homeomorphism $f:X\rightarrow X$ and we integrated maps $h:X\rightarrow\mathbb{Z}$ of the form $h=\sum_{j=1}^n d_j\mathds{1}_{U_j}$, where the $U_j$'s were $f$-invariant definable subsets of $X$. In order to compute the integral, there were two main steps to solve. The first one was to obtain the triangulation of the Triangulation Theorem, and the second one was to obtain a simplicial approximation of the homeomorphisms.

In this paper, for any $f:X\rightarrow X$ (we don't require $f$ to be a homeomorphism), we are going to integrate maps $h:X\rightarrow \mathbb{Z}$ of the form $h=\sum_{j=1}^n d_j\mathds{1}_{U_j}$ where the only condition of the $U_j$'s is that $f$ has no fixed points in $\overline{U_j}\setminus \mathring{U_j}$. This improvement will be possible since we are going to integrate with respect to the combinatorial fixed point index, which in the case where $f$ has no fixed points in $\overline{U_j}\setminus \mathring{U_j}$ is a generalization of the combinatorial Lefschetz number (Theorem~\ref{igualdad indice numero combinatorio}). Moreover, this constitues a good measure of the ``density'' and ``weight'' of the fixed points.
\begin{definition}
    Let $X$ be a simplicial complex and $f:X\rightarrow X$ a continuous map. A map $h:X\rightarrow \mathbb{Z}$ is called $f$-\textit{integrable with respect to the index} if it admits an expression of the form
    \begin{equation*}
        h=\sum_{j=i}^n d_j\mathds{1}_{U_j},
    \end{equation*}
    where $f$ has no fixed points in $\overline{U_j}\setminus \mathring{U_j}$ for each $j\in\{1,\ldots, n\}$.
\end{definition}
\begin{definition}
    Let $X$ be a simplicial complex, $f:X\rightarrow X$ a continuous map and let $h=\sum_{j=i}^n d_j\mathds{1}_{U_j}:X\rightarrow\mathbb{Z}$ be $f$-integrable with respect to the index. We define the integral of $h$ \textit{with respect to} \textit{the combinatorial fixed point index of} $f$ as
    \begin{equation*}
        \int_X h\; d\,i_c(f)= \sum_{j=1}^n d_j\cdot i_c(X,f,U_j).
    \end{equation*}
\end{definition}
We must check that this integral is well defined. The proof is inspired by \cite[Theorem 5.2]{M-M1}.
\begin{theorem}\label{buena def int}
    Let $X$ be a simplicial complex, $f:X\rightarrow X$ a continuous map and let $h=\sum_{j=i}^n d_j\mathds{1}_{U_j}=\sum_{k=i}^m e_k\mathds{1}_{V_k}$, where $f$ has no fixed points in $\overline{U_j}\setminus \mathring{U_j}$ for $j\in\{1,\ldots n\}$ nor in $\overline{V_k}\setminus V_k$ for $k\in \{1,\ldots, m\}$. Then
    \begin{equation*}
        \sum_{j=1}^n d_j\cdot i_c(X,f,U_j)=\sum_{k=1}^m e_k\cdot i_c(X,f,V_k).
    \end{equation*}
\end{theorem}
\begin{proof}
    We start by writing $h$ as a linear combination of characteristic maps of sets compatibles with both the $U_j$'s and the $V_k$'s.

    Let $W_\zeta=U_\zeta$ if $1\le\zeta\le n$ and $W_\zeta=V_{\zeta-n}$ if $n+1\le\zeta\le n+m$. We divide $(\cup_{j=1}^n U_j)\cup (\cup_{k=1}^m V_k)$ into the following disjoint subsets (indexed by $s=1,\ldots,\mu$): 
\begin{align}\label{conjuntos teor integral}
&L_1=\bigcap_{\zeta=1}^{n+m}W_{\zeta},\;
L_2=\bigcap_{\zeta\neq \zeta_1}W_{\zeta}\setminus \bigcap_{\zeta=1}^{n+m}W_{\zeta},\\
&L_3=\bigcap_{\zeta\neq \zeta_2}W_{\zeta}\setminus \bigcap_{\zeta=1}^{n+m}W_{\zeta}\setminus \Big(\Big(\bigcap_{\zeta\neq \zeta_1}W_{\zeta}\Big)\setminus \Big(\bigcap_{\zeta=1}^{n+m}W_{\zeta}\Big)\Big),\;\ldots \nonumber
\end{align}

Let $\mathcal{B}$ be the family of subsets $A\subset X$ such that $f$ has no fixed points in $\overline{A}\setminus \mathring{A}$. Note that $\mathcal{B}$ is closed under finite intersections and complements, so all the elements in (\ref{conjuntos teor integral}) belong to $\mathcal{B}$. Moreover, if we write 
\begin{align*}
&N_j=\{s\in \{1,\ldots,\mu\} \;:\; L_s\subset U_j\},\\
&M_k=\{s\in \{1,\ldots,\mu\} \;:\; L_s\subset V_k\}
\end{align*}
and we apply the Additivity Axiom of the combinatorial fixed point index we have:
\begin{align*}
    & \sum_{j=i}^n d_j\cdot i_c(X,f,U_j)=\sum_{j=1}^n d_j \sum_{s\in N_j}i_c(X,f,L_s)\\
    &=\sum_{s=1}^{\mu} i_c(X,f,L_s) \cdot\sum_{j:s\in N_j} d_j
    =\sum_{s=1}^{\mu} i_c(X,f,L_s)\cdot\sum_{k:s\in M_k} e_k\\
    &= \sum_{k=1}^m e_k\sum_{s\in M_k}i_c(X,f,L_s)=\sum_{k=1}^m e_k\cdot i_c(X,f,V_k),
\end{align*}
where the third equality follows from $\sum_{j:s\in N_j}d_j=\sum_{k:s\in M_k}e_k$ for $s=1,\ldots, \mu$ (which is a consequence of $\sum_{j=i}^n d_j\mathds{1}_{U_j}=\sum_{k=i}^m e_k\mathds{1}_{V_k}$).

With this, we conclude the result.
\end{proof}
This integral satisfies some important properties.
\begin{theorem}
    Let $X$ be a simplicial complex, $f:X\rightarrow X$ a continuous map and let $h:X\rightarrow\mathbb{Z}$ be $f$-integrable with respect to the index. Then,
    \begin{equation*}
        \int_X h\;d\,i_c(f)=\sum_{p\in\mathbb{Z}}p\cdot i_c(X,f,\{h=p\}),
    \end{equation*}
    where $\{h=k\}\coloneqq \{x\in X : h(x)=k\}$.
\end{theorem}
\begin{proof}
    Since the family $\mathcal{B}$ defined in Theorem~\ref{buena def int} is closed under finite unions, the result follows from the Additivity Axiom of the combinatorial fixed point index and from
    \begin{equation*}
        \sum_{j=i}^n d_j\cdot i_c(X,f,U_j)=\sum_{s=1}^{\mu} i_c(X,f,L_s) \cdot\sum_{j:s\in N_j} d_j
    \end{equation*}
    (see Theorem~\ref{buena def int}).
\end{proof}
\begin{theorem}[Product rule of the combinatorial fixed point index]\label{regla producto}
    Let $X_1$, $X_2$ be simplicial complexes, $f_1:X_1\rightarrow X_1$, $f_2:X_2\rightarrow X_2$ continuous maps and $A_1\subset X_1$, $A_2\subset X_2$ subspaces such that $f_1$ has no fixed points in $\overline{A_1}\setminus \mathring{A_1}$ and $f_2$ has no fixed points in $\overline{A_2}\setminus \mathring{A_2}$. Then, 
    \begin{equation*}
        i_c(X_1\times X_2,f_1\times f_2,A_1\times A_2)=i_c(X_1,f_1,A_1)\cdot i_c(X_2,f_2,A_2).
    \end{equation*}
\end{theorem}
\begin{proof}
    The theorem follows from the product rule of the fixed point index \cite[IV.B Theorem 6]{Brown}. First, note that since $f_1$ has no fixed points in $\overline{A_1}\setminus \mathring{A_1}$ and $f_2$ has no fixed points in $\overline{A_2}\setminus \mathring{A_2}$, the $f_1\times f_2$ has no fixed points in $\overline{A_1\times A_2}\setminus \mathring{(A_1\times A_2)}$.
    Now, applying the product rule of the fixed point index \cite[IV.B Theorem 6]{Brown}, we obtain
    \begin{align*}
        &i_c(X_1\times X_2,f_1\times f_2,A_1\times A_2)=i(X_1\times X_2,f_1\times f_2,\mathring{(A_1\times A_2)})\\
        &=i(X_1\times X_2,f_1\times f_2,\mathring{A_1}\times \mathring{A_2})=i(X_1,f_1,\mathring{A_1})\cdot i(X_2,f_2,\mathring{A_2})\\
        &=i_c(X_1,f_1,A_1)\cdot i_c(X_2,f_2,A_2). \qedhere
    \end{align*}
\end{proof}

The previous product rule admits, like in \cite[Lemma 6.2]{M-M1}, a generalization that, in fact, is a primitive Fubini's theorem. 

\begin{theorem}\label{especie de fubini}
    Let $(X,p,B)$ be a fiber bundle with typical fiber $F$, where $p$ is definable and $X$, $B$ and $F$ are simplicial complexes. Let $B'\subset B$ be definable such that, if we denote $A=p^{-1}(B')$ we have $\overline{A}\overset{g}{\approx} \overline{B'}\times F $ and $A\overset{g}{\approx} B'\times F$. Let $h:X\rightarrow \mathbb{Z}$ be defined as $h=\mathds{1}_A$ and let $l:X\rightarrow X$ be a map such that $l_{|\overline{A}}\equiv l_1\times l_2$ with $l_1:\overline{B'}\rightarrow\overline{B'}$ and $l_2:F\rightarrow F$. Suppose also that $l_1$ has no fixed points in $\overline{B'}\setminus\mathring{B'}$ (and, hence, $l$ has no fixed points in $\overline{A}\setminus \mathring{A}$). Then
    \begin{equation*}
        \int_X h\; d\, i_c(l)= \int_B\biggl(\int_{p^{-1}(b)} h\;d\,i_c(\mathrm{id}\times l_2)\biggr)\;d\,i_c(l_1).
    \end{equation*}
\end{theorem}
\begin{proof}
    We divide the proof in two cases. If $l$ has no fixed points in $\overline{A}$ then, by Corrollary~\ref{pto fijo indice comb} both terms are zero. In other case, we can consider $B''\subset B'$ such that all the fixed points of $l_1$ in $B'$ are in $(B'')^{\circ}$ (and so, all the fixed point of $l$ in $\overline{A}$ are in $(p^{-1}(B''))^{\circ}$) and such that $l(\overline{p^{-1}(B'')})\subset A$.
    
    Consider now a triangulation $(Y,\overline{W},W,V)$ that is compatible with $(X,\overline{p^{-1}(B')},p^{-1}(B'),p^{-1}(B''))$.
    By definition, we have
    \begin{equation*}
        \int_X \mathds{1}_{A}\;d\,i_c(l)=i_c(X,l,p^{-1}(B')).
    \end{equation*}
    Now, by the Additivity Axiom of the combinatorial fixed point index, the last term equals $i_c(X,l,p^{-1}(B''))$. Now we have
    \begin{equation*}
        i_c(X,l,p^{-1}(B''))=i(X,l,\mathring{p^{-1}(B'')})=i(Y,\tilde{l},\mathring{V}),
    \end{equation*}
    where $\tilde{l}$ is the map induced by $l$ trough the triangulation. But by \cite[Corollary 7.2]{ONeill}, we obtain
    \begin{equation*}
        i(Y,\tilde{l},\mathring{V})= i(B\times F,l_1\times l_2, \mathring{(B''\times F)})=i(B,l_1,\mathring{B''})\cdot i(F,l_2,F),
    \end{equation*}
    where the last equality follows from the product rule (Theorem~\ref{regla producto}). From here if we take $b_0\in B''$, we get
    \begin{align*}
        &i(B,l_1,\mathring{B''})\cdot i(F,l_2,F)=i_c(B,l_1,B'')\cdot i_c(F,l_2,F)\\
        &=i_c(B,l_1,B'')\cdot \int_F \mathds{1}_F\; d\, i_c(l_2)=i_c(B,l_1,B'')\cdot \int_{p^{-1}(b_0)} \mathds{1}_{p^{-1}(b_0)} \; d\, i_c(\mathrm{id}\times l_2)\\
        &=\int_B \mathds{1}_{B''} \biggl(\int_{p^{-1}(b_0)} \mathds{1}_{p^{-1}(b_0)}\; d\, i_c(\mathrm{id}\times l_2)\biggr)\;d\, i_c(l_1)\\
        &=\int_{B} \biggl(\int_{p^{-1}(b)} \mathds{1}_{p^{-1}(B'')}\; d\, i_c(\mathrm{id}\times l_2)\biggr)\;d\,i_c(l_1)\\
        &=\int_{B} \biggl(\int_{p^{-1}(b)} \mathds{1}_{p^{-1}(B'')}\; d\, i_c(\mathrm{id}\times l_2)\biggr)\;d\,i_c(l_1)+0\\
        &=\int_{B} \biggl(\int_{p^{-1}(b)} \mathds{1}_{p^{-1}(B'')}\; d\, i_c(\mathrm{id}\times l_2)\biggr)\;d\,i_c(l_1)\\
        &+\int_{B} \biggl(\int_{p^{-1}(b)} \mathds{1}_{A\setminus p^{-1}(B'')}\; d\, i_c(\mathrm{id}\times l_2)\biggr)\;d\,i_c(l_1)\\
        &=\int_{B} \biggl(\int_{p^{-1}(b)} h\; d\, i_c(\mathrm{id}\times l_2)\biggr)\;d\,i_c(l_1). \qedhere
    \end{align*}
\end{proof}
\begin{remark}
    Note that in Theorem~\ref{especie de fubini} we do not require $l$, $l_1$ and $l_2$ to be homeomorphisms. The hypothesis of $l$-invariance needed in \cite[Lemma 6.2]{M-M1} is also removed.
\end{remark}

\section{Integration of real maps}\label{int real}

In \cite{B-G2}, the authors introduced the concepts of lower and upper Riemann-sum of a real valued function h from a simplicial complex $X$ with respect to the combinatorial Euler characteristic. The purpose of this section is to extend this notion to a sort of Riemann-sums with respect to the fixed points. A first approach to do this could be to integrate like in \cite{B-G2} but considering the combinatorial Lefschetz number of a map $f:X\rightarrow X$ instead of the combinatorial Euler characteristic. Nevertheless, as we take smaller partitions of an interval, all the level sets of these elements of the partition should be $f$-invariant, which is an enormous restriction on the maps that can be integrated.

To solve this problem, the main idea is to integrate with respect to the combinatorial fixed point index, instead of the combinatorial Lefschetz number. Indeed, with this tool, we do not need to restrict the integration to homeomorphisms.
\begin{definition}
    Let $X$ be a simplicial complex and $f:X\rightarrow X$ a continuous map. We say that $f$ is \textit{index-strict} if one of the followings properties hold:
    \begin{itemize}
        \item For all open set $U\subset X$ such that $f$ has no fixed points in $\overline{U}\setminus \mathring{U}$, $i(X,f,U)\geq 0$.
        \item For all open set $U\subset X$ such that $f$ has no fixed points in $\overline{U}\setminus \mathring{U}$, $i(X,f,U)\leq 0$.
    \end{itemize}
\end{definition}

\begin{definition}\label{def real integrable}
    Let $X$ be a simplicial complex and $f:X\rightarrow X$ a index-strict map. A map $h:X\rightarrow\mathbb{R}$ is called \textit{real} $f$-\textit{integrable} if the following properties hold:
    \begin{itemize}
    \item For all rational $p$, $f$ has no fixed points in $\overline{h^{-1}(p)}$.
    \item For all pair $(p,q)$ of rational numbers ($p<q$) $f$ has no fixed points in $\overline{h^{-1}((p,q))}\setminus \mathring{h^{-1}((p,q))}$.
    \item The map $h$ is bounded.
    \end{itemize}
\end{definition}
\begin{remark}
    In Definition~\ref{def real integrable} we could have asked for the first two conditions in some other dense subsets of $\mathbb{R}$ instead of in $\mathbb{Q}$. However, in that case, Definition~\ref{def sumas riemann} would not admit such a clear expression. 
\end{remark}
\begin{lemma}
    For a continuous map $h$, if $f$ has no fixed points in $h^{-1}(p)$ for $p\in\mathbb{Q}$, then it is real $f$-integrable.
\end{lemma}
\begin{proof}
    Since $X$ is compact and $h$ is continuous, then $h$ is bounded. Consider now $p,q\in \mathbb{Q}$ ($p<q$). We must check that $f$ has no fixed points in $\overline{h^{-1}(p,q)}\setminus \mathring{h^{-1}(p,q)}$. Now,
    \begin{align*}
        &\overline{h^{-1}((p,q))}\setminus \mathring{h^{-1}((p,q))}=\overline{h^{-1}((p,q))}\setminus h^{-1}((p,q))\subset h^{-1}([p,q])\setminus h^{-1}((p,q))\\
        & =h^{-1}(\{p,q\})=h^{-1}(\{p\})\cup h^{-1}(\{q\}).
    \end{align*}
    Hence, since $f$ has no fixed points in $h^{-1}(\{p\})$ and $h^{-1}(\{q\})$, neither does in $\overline{h^{-1}(p,q)}\setminus \mathring{h^{-1}(p,q)}$.
\end{proof}
\begin{definition}\label{def sumas riemann}
    Let $X$ be a simplicial complex, $f:X\rightarrow X$ a index-strict map and $h:X\rightarrow\mathbb{R}$ a real $f$-integrable map. We define the \textit{lower Riemann-sum} \textit{with respect} \textit{to the combinatorial fixed point index} as
    \begin{equation}\label{suma inf}
        \int_X h \;\lfloor d\,i_c(f)\rfloor=\mathrm{lim}_{n\rightarrow \infty} \frac{1}{2^n}\int_X \lfloor 2^nh\rfloor\;d\,i_c(f),
    \end{equation}
    where $\lfloor\cdot\rfloor$ denotes the floor function. Similarly, we define the \textit{upper Riemann-sum} \textit{with respect} \textit{to the combinatorial fixed point index} as
    \begin{equation}\label{suma sup}
        \int_X h\; \lceil d\,i_c(f)\rceil=\mathrm{lim}_{n\rightarrow \infty} \frac{1}{2^n}\int_X \lceil 2^nh\rceil\;d\,i_c(f).
    \end{equation}
    (Here $\lceil\cdot\rceil$ denote the ceiling function).
    \end{definition}
    As in \cite{B-G2}, both sums might not agree. However, it can be proved that these sums converge.
    \begin{theorem}
        Let $X$ be a simplicial complex, $f:X\rightarrow X$ a index-strict map and $h:X\rightarrow\mathbb{R}$ a real $f$-integrable map. Then the limits presented in Equations~(\ref{suma inf}) and (\ref{suma sup}) exist.
    \end{theorem}
\begin{proof}
    We will show the convergence of the lower Riemann-sum. Let $M_1$ and $M_2$ be respective lower and upper integer bounds of $h$ (remember that $h$ is bounded since it is real $f$-integrable). It is easy to verify that
    \begin{equation*}
        \frac{1}{2^n}\int_X \lfloor 2^nh\rfloor\;d\,i_c(f)=\sum_{m=M_1}^{M_2} \sum_{l=0}^{2^n-1}( m+\frac{l}{2^n})\cdot i_c\bigl(X,f,h^{-1}((m+\frac{l}{2^n},m+\frac{l+1}{2^n}))\bigr)
    \end{equation*}
    (remember that, since $f$ has no fixed points in the inverse images of rational numbers, $i_c(X,f,[m+\frac{l}{2^n},m+\frac{l+1}{2^n}))=i_c(X,f,(m+\frac{l}{2^n},m+\frac{l+1}{2^n}))$).

    Since $f$ is index-strict, imagine for example that for all open set $U\subset X$ such that $f$ has no fixed points in $\overline{U}\setminus \mathring{U}$, $i(X,f,U)\geq 0$ (the other case is analogous). As a consequence of this positivity of the index and of the Additivity Axiom of the combinatorial fixed point index we have
    \begin{align*}
        &(m+\frac{l}{2^n})\cdot i_c\bigl(X,f,h^{-1}((m+\frac{l}{2^n},m+\frac{l+1}{2^n}))\bigr)\\
       & =(m+\frac{l}{2^n})\cdot\biggl(   i_c\bigl(X,f,h^{-1}((m+\frac{l}{2^n},m+\frac{l}{2^n}+\frac{1}{2^{n+1}}))\bigr) \\
       &+i_c\bigl(X,f,h^{-1}((m+\frac{l}{2^n}+\frac{1}{2^{n+1}},m+\frac{l+1}{2^n}))\bigr)  \biggr)\\
       &\leq(m+\frac{l}{2^n})\cdot   i_c\bigl(X,f,h^{-1}((m+\frac{l}{2^n},m+\frac{l}{2^n}+\frac{1}{2^{n+1}}))\bigr) \\
       &+(m+\frac{l}{2^n}+\frac{1}{2^{n+1}})\cdot i_c\bigl(X,f,h^{-1}((m+\frac{l}{2^n}+\frac{1}{2^{n+1}},m+\frac{l+1}{2^n}))\bigr) \\
       &=(m+\frac{2l}{2^{n+1}})\cdot   i_c\bigl(X,f,h^{-1}((m+\frac{2l}{2^{n+1}},m+\frac{2l+1}{2^{n+1}}))\bigr)\\
       &+(m+\frac{2l+1}{2^{n+1}})\cdot i_c\bigl(X,f,h^{-1}((m+\frac{2l+1}{2^{n+1}},m+\frac{2l+2}{2^{n+1}}))\bigr)
    \end{align*}
    for all $n$, $l\in\{1,\ldots, 2^n-1\}$ and $m\in\{M_1,\ldots,M_2\}$.

    It follows then that the sequence $\frac{1}{2^n}\int_X \lfloor 2^nh\rfloor\;d\,i_c(f)$ is increasing and bounded by 
    \begin{equation*}
        \sum_{m=M_1}^{M_2}(m+1)\cdot i_c\bigl(X,f,h^{-1}((m,m+1))\bigr),
    \end{equation*}
    so it converges. Note that if the index $i(X,f,\_)$ were always negative, the sequence would decrease and have a lower bound.
\end{proof}

The following classical result follows from the Commutative Axiom of the fixed point index.
\begin{lemma}\label{ultimo lema auxiliar}
    Given $X$ and $Y$ simplicial complexes, $f:X\rightarrow X$ a continuous map, $g:X\rightarrow Y$ a homeomorphism and $A\subset X$ a open subset such that $f$ has no fixed points in $\overline{A}\setminus A$ then 
    \begin{equation*}
    i(X,f,A)=i(Y,g\circ f \circ g^{-1}, g(A)).
    \end{equation*}

\end{lemma}

The next result generalizes \cite[Lemma 2]{B-G2}.
\begin{proposition}[Robustness under changes of coordinates]
    Let $X$ and $Y$ be simplicial complexes, $f:X\rightarrow X$ a index-strict map, $h:X\rightarrow\mathbb{R}$ a real $f$-integrable map and $g:X\rightarrow Y$ a homeomorphism. Then:
    \begin{equation*}
        \int_X h \lfloor d\,i_c(f)\rfloor=\int_Y h\circ g^{-1} \lfloor d\,i_c(g\circ f\circ g^{-1})\rfloor
    \end{equation*}
    and
        \begin{equation*}
        \int_X h \lceil d\,i_c(f)\rceil=\int_Y h\circ g^{-1} \lceil d\,i_c(g\circ f\circ g^{-1})\rceil.
    \end{equation*}
\end{proposition}   
\begin{proof}
    Due to Lemma~\ref{ultimo lema auxiliar} we have
\begin{equation*}
    \frac{1}{2^n}\int_X \lfloor 2^nh\rfloor\;d\,i_c(f)=\frac{1}{2^n}\int_Y \lfloor 2^nh\circ g^{-1}\rfloor\;d\,i_c(g\circ f\circ g^{-1})
\end{equation*}
for each $n\in \mathbb{N}$ (similarly for the upper Riemann-sums). The result then follows from Definition~\ref{def sumas riemann}.
\end{proof}


\begin{thebibliography}{99}

\bibitem{A-B} M. Arkowitz, R. Brown (2004). \emph{The Lefschetz-Hopf theorem and axioms for the Lefschetz number}, Fixed Point
Theory Appl., (1), 1--11.
 \bibitem{B-G1}  Y. Baryshnikov, R. Ghrist (2009).\emph{Target enumeration via Euler characteristic integrals}, SIAM Journal on Applied Mathematics, (70), \textbf{3}, 825--844.
\bibitem{B-G2} Y. Baryshnikov, R. Ghrist (2010). \emph{Euler integration over definable functions}, Proc, Natl. Acad. Sci. USA, (107), \textbf{21}, 9525--9530.
\bibitem{Brown} R. F. Brown (1971). \emph{The Lefschetz Fixed Point Theorem}, Scott, Foresman and Company.
\bibitem{C-G-R} J. Curry, R. Ghrist and M. Robinson (2012). \emph{Euler calculus with applications to signals and sensing}, Proceedings of Symposia in Applied Mathematics, (70), 75--146.
 \bibitem{M-M1} A. O. Majadas-Moure, D. Mosquera-Lois (2025). \emph{Integration with respect to the Lefschetz number}.  J. Fixed Point Theory Appl. 27, 51.
 \bibitem{M-M2} A. O. Majadas-Moure, D. Mosquera-Lois (2024). \emph{Sheaf Theoretic Approach to Lefschetz Calculus}. arXiv:2410.13653.
\bibitem{McCrory} C. McCrory,  A. Parusinski (1997). \emph{Algebraically constructible functions}, Ann. Sci. École Norm. Sup., (4), \textbf{30}, 527--552.
\bibitem{ONeill} B. O'Neill (1953). \emph{Essential sets and fixed points}, Amer. J. Math., \textbf{75}, 497--509.
\bibitem{Dries} L. P. D. van der Dries  (2009). \emph{Tame Topology and O-minimal Structures}, Cambridge University Press.
        
        
       

        
       
     
		%
	\end{thebibliography}
 \end{document}